\documentclass[12pt,a4paper]{article}
\usepackage[utf8]{inputenc}
\usepackage{lmodern}
\usepackage[T1]{fontenc}
\usepackage[english]{babel}
\usepackage{ifpdf}
\usepackage[usenames,dvipsnames]{xcolor}
\usepackage{graphicx}
\usepackage[shortlabels]{enumitem}
\usepackage[font={small,it}]{caption}
\usepackage[parfill]{parskip}
\usepackage{hyperref}
\usepackage[top=2cm,bottom=2.5cm]{geometry}
\usepackage[setbuildercolon]{matmatix}

\ifpdf
  \hypersetup{
    pdfauthor = {François Bienvenu}, 
    pdftitle = {Positive association of the oriented percolation cluster
                in randomly oriented graphs}
  }
\fi

\hypersetup{colorlinks, allcolors = {MidnightBlue}}

\newenvironment{mathlist}
{\begin{enumerate}[label={\upshape(\roman*)}, align=left, widest=iii, leftmargin=*]}
{\end{enumerate}\ignorespacesafterend}

\begingroup
 \makeatletter
 \@for\theoremstyle:=definition,remark,plain\do{%
  \expandafter\g@addto@macro\csname th@\theoremstyle\endcsname{%
  \addtolength\thm@preskip\parskip }%
 }
\endgroup

\newtheorem{theorem}{Theorem}[section]
\newtheorem{lemma}[theorem]{Lemma}

\newtheorem{corollary}[theorem]{Corollary}

\theoremstyle{definition}
\newtheorem{definition}{Definition}[section]
\newtheorem*{remark*}{Remark}

\newcommand{\hardcoded}[1]{#1}

\newcommand{\leadstoin}[1]{\overset{\scriptscriptstyle #1}{\leadsto}}
\newcommand{\dleadstoin}[1]{\overset{#1}{\leadsto}}
\newcommand{\toin}[1]{\overset{\scriptscriptstyle #1}{\to}}

\title{\hardcoded{\vspace{-3em}}Positive association of the oriented
percolation cluster in randomly oriented graphs}

\author{François Bienvenu\,\thanks{\scriptsize Center for Interdisciplinary
Research in Biology (CIRB), CNRS UMR 7241, Collège de France, Paris, France.}
\thanks{\scriptsize Laboratoire de Probabilités, Statistique et Modélisation (LPSM),
CNRS UMR 8001, Sorbonne Université, Paris, France\newline
E-mail: \href{mailto:francois.bienvenu@normalesup.org}{\sf francois.bienvenu@normalesup.org}}}

\begin{document}

\maketitle

\begin{abstract}
Consider any fixed graph whose edges have been randomly and independently
oriented, and write $\{S \leadsto i\}$ to indicate that there is an oriented
path going from a vertex $s \in S$ to vertex~$i$. Narayanan (2016) proved that
for any set $S$ and any two vertices $i$ and $j$, $\{S \leadsto i\}$ and $\{S
\leadsto j\}$ are positively correlated. His proof relies on the
Ahlswede-Daykin inequality, a rather advanced tool of probabilistic
combinatorics.
  
In this short note, I give an elementary proof of the following, stronger
result: writing $V$ for the vertex set of the graph, for any source set $S$,
the events $\{S \leadsto i\}$, $i \in V$, are positively associated -- meaning
that the expectation of the product of increasing functionals of the
family $\{S \leadsto i\}$ for $i \in V$ is greater than the product of their
expectations.
\end{abstract}

\section{Introduction}

Oriented percolation is the study of connectivity in a random oriented graph.
In most settings, one starts from a graph with a fixed orientation and then
keeps each edge with a given probability. Classical such models include
the north-east lattice~\cite{DurrettAnnProb1984} and the
hypercube~\cite{FillAnnApplProb1993}.

Another broad and natural class of random oriented graphs is
obtained by starting from a fixed graph and then orienting each edge,
independently of the orientations of other edges. Note that, in the general
case, the orientations of the edges need not be unbiased: some edges can be
allowed to have a higher probability to point towards one of their ends than
towards the other. Percolation on such \emph{randomly oriented graphs}
has been studied, e.g.\ in~\cite{LinussonArXiv2009}, and more recently in
\cite{NarayananCombinatProbComput2016}, which motivated the present work.

In \cite{NarayananCombinatProbComput2016}, Narayanan showed that if the edges
of any fixed graph are randomly and independently oriented, then writing
$\{S \leadsto i\}$ to indicate that there is an oriented path going from a
vertex $s \in S$ to vertex~$i$, we have
\[
  \Prob{S \leadsto i, S \leadsto j} \;\geq\;
  \Prob{S \leadsto i}\, \Prob{S \leadsto j} \,.
\]
The aim of this note is to strengthen and simplify the proof of this result.
More specifically, let $V$ be the vertex set of the graph. We prove that the
events $\{S \leadsto i\}$, $i \in V$, are positively associated, without
resorting to advanced results such as the Ahlswede--Daykin
inequality~\cite{AhlswedeDaykinProbTheorRF1978}.

\subsection{Positive association and related notions}

There are many ways to formalize the idea of a positive dependence between
the random variables of a family $\mathbf{X} = (X_i)_{i \in I}$.
A straightforward, weak one is to ask that these variables be pairwise
positively correlated, i.e.
\[
  \forall i, j \in I, \quad
  \Expec{X_i X_j} \;\geq\; \Expec{X_i}\, \Expec{X_j} \, .
\]
A much stronger condition, due to~\cite{EsaryAnnMathStat1967},
is known as positive association.  In the following definition and throughout
the rest of this note, we use bold letters to denote vectors, as in
$\mathbf{X} = (X_i)_{i \in I}$, and we write $\mathbf{X} \leq \mathbf{X}'$ to
say that $X_i \leq X'_i$ for all $i$.  Finally, a function
$f\colon \R^I \to \R$ is said to be \emph{increasing} when
$\mathbf{X} \leq \mathbf{X'} \implies f(\mathbf{X}) \leq f(\mathbf{X}')$.

\begin{definition} \label{defPositiveAsso}
The random vector $\mathbf{X} = (X_i)_{i \in I}$ is said to be
\emph{positively associated} when, for all increasing functions $f$ and $g$,
\[
  \Expec{f(\mathbf{X})g(\mathbf{X})} \;\geq\;
  \Expec{f(\mathbf{X})}\,\Expec{g(\mathbf{X})}
\]
whenever these expectations exist.
\end{definition}

Without further mention, we only consider test functions $f$ and $g$ for which
$\Expec{f(\mathbf{X})}$, $\Expec{g(\mathbf{X})}$ and
$\Expec{f(\mathbf{X})g(\mathbf{X})}$ exist.

We say that the events $A_i$, $i \in I$, are positively associated when the
corresponding vector of indicator variables $(\Indic*{A_i})_{i \in I}$ is
positively associated. Similarly, a random subset $R$ of the
fixed set $I$ can be seen as the vector
\[
  \mathbf{R} = \big(\Indic{i \in R}\big)_{i \in I}\,, 
\]
so that $R$ is said to be positively associated when the events $\Set{i \in R}$,
$i \in I$, are. This is equivalent to saying that for any
increasing functions $f$ and $g$ from the power set of $I$ to $\R$,
\[
  \Expec{f(R)g(R)} \;\geq\; \Expec{f(R)}\,\Expec{g(R)}\,,
\]
where $f$ being increasing is understood to mean that
${r' \subset r \implies f(r') \leq f(r)}$.

Positive association is famous for the FKG theorem, which states that it is
implied by a lattice condition that can sometimes be very easy to
check~\cite{FortuinCommMathPhy1971}. Another reason why it is so useful is
that it implies weaker positive dependence notions that have to be checked in
applications. One example of this is the existence of increasing couplings
and the corresponding notion of \emph{positive relation} used in the
Stein--Chen method -- see e.g, \cite{Barbour1992}
and~\cite{RossProbabilitySurveys2011}.

\subsection{Notation}

Let us fix some notation to be used throughout the rest of this document.

We study the simple graph $G = (V, E)$. Unless explicitly specified otherwise,
$V$ is assumed to be finite and we denote by $|V|$ its cardinality.
The edges of $G$ have a random orientation that is independent of the
orientations of other edges and we write $\Set{i \to j}$ to indicate
that the edge~$\{ij\}$ is oriented towards~$j$. Formally, we are
thus given a family of events $(\Set{i \to j},\, \{ij\} \in E)$ such that
$\Set{i \to j} = \Complement*{\Set{j \to i}}$ and 
for all $\{ij\}$, $\{i \to j\} \independent (\{k \to \ell\}, \{k\ell\} \neq \{ij\})$.

Finally, for every pair of vertices $i$ and $j$, we write
$\Set{i \leadsto j}$ for the event that there exists an oriented path going
from $i$ to $j$. Similarly, for every source set $S$ we let
$\Set{S \leadsto i} = \bigcup_{j \in S} \Set{j \leadsto i}$ be the event
that there is an oriented path from $S$ to $i$, and for
every target set $T$ we let
$\Set{i \leadsto T} = \bigcup_{j \in T} \Set{i \leadsto j}$ be the event
that there is an oriented path from $i$ to $T$.
If there is an ambiguity regarding which graph is considered for these
events, we will specify it with the notation $\Set*{i \leadstoin{G} j}$.

\section{Positive association of the percolation cluster}

\subsection{Preliminary lemma}

\begin{lemma} \label{lemmapositiveassociation}
Let $\Gamma$ be a finite set and let $R$ be a positively associated
random subset of $\Gamma$.
Let $X^r_i$, $r \subset \Gamma$ and $i \in V$, be a family of
events on the same probability space as $R$ with the property that
\begin{mathlist}
\item $r' \subset r \implies  X^{r'}_i \subset X^{r}_i$, $\forall i \in V$.
\item For all $r \subset \Gamma$, $(X^r_i)_{i \in V}$ is positively
associated and independent of $R$.
\end{mathlist}
For all $i \in V$, define $X^{R}_i$ by
\[
  X^R_i \defas \bigcup_{r \subset \Gamma} \Set{R = r} \cap X^r_i \,.
\]
Then, the events $X^R_i$, $i \in V$, are positively associated.
\end{lemma}

\begin{proof}
Let $f$ and $g$ be two increasing functions. We have
\begin{align*}
  \Expec{f(\mathbf{X}^R) g(\mathbf{X}^R)}
  &= \sum_{r \subset \Gamma}
     \Expec*{\big}{f(\mathbf{X}^r)g(\mathbf{X}^r) \Indic{R = r}} \\
  &= \sum_{r \subset \Gamma} 
     \Expec*{\big}{f(\mathbf{X}^r)g(\mathbf{X}^r)}\, \Prob{R = r}  \\
  &\geq \sum_{r \subset \Gamma}
     \Expec{f(\mathbf{X}^r)}\, \Expec{g(\mathbf{X}^r)} \,  \Prob{R = r}\, , 
\end{align*}
because $\mathbf{X}^r \independent R$
and $\mathbf{X}^r$ is positively associated.
Now, let $u \colon r \mapsto \Expec{f(\mathbf{X}^r)}$ and
$v \colon r \mapsto \Expec{g(\mathbf{X}^r)}$, so that the last sum
is $\Expec{u(R) v(R)}$. Note that $u$ and $v$ are increasing, since $f$ and
$g$ are and, by hypothesis,
$r' \subset r \implies \mathbf{X}^{r'} \leq \mathbf{X}^r$.
Therefore, by the positive association of $R$,
\begin{align*}
  \Expec*{\big}{u(R) v(R)} \;
  &\geq\; \Expec{u(R)}\, \Expec{v(R)} \, .
\end{align*}
Finally, using again the independence of $\mathbf{X}^r$ and $R$,
we have $\Expec{u(R)} = \Expec*{\normalsize}{f(\mathbf{X}^R)}$ and
$\Expec{v(R)} = \Expec*{\normalsize}{g(\mathbf{X}^R)}$, 
which concludes the proof.
\end{proof}

\subsection{Main result}

\begin{theorem} \label{thmMainResult}
Let $G$ be a finite graph with vertex set $V$, whose edges have been
randomly and independently oriented. Then, for any source set $S$,
the events $\Set{S \leadsto i}$, $i \in V$, are positively associated, i.e.,
for all increasing functions $f$ and $g$ and writing
$\mathbf{X} = (\Indic{S \leadsto i})_{i \in V}$,
\[
  \Expec{f(\mathbf{X}) g(\mathbf{X})} \;\geq\;
  \Expec{f(\mathbf{X})}\, \Expec{g(\mathbf{X})} \, .
\]
\end{theorem}

\begin{proof}
Our proof uses the same induction on the number of vertices as Narayanan's.
The difference is that we use Lemma~\ref{lemmapositiveassociation} rather
than the Ahlswede--Daykin inequality to propagate the positive dependence.

The theorem is trivial for the graph consisting of a single vertex (a family
of a single variable being always positively associated) so let us assume that
it holds for every graph with strictly less than $\Abs{V}$ vertices.

Let $\Gamma$ be the neighborhood of $S$, i.e.\
\[
  \Gamma = \Set*[\big]{v \in V \setminus S \suchthat \exists s \in S \st \{vs\} \in E}\,.
\]
Then, let $R$ be the random subset of $\Gamma$ defined by
\[
  R = \Set*[\big]{v \in \Gamma \suchthat \exists s \in S \st s \to v}\,.
\]
Observe that the events $\Set{i \in R}$, $i \in \Gamma$ are independent, so that
the set $R$ is positively associated.

Next, let $H$ be the subgraph of $G$ induced by $V \setminus S$. Note that, for
all $i \in V \setminus S$,
\[
  \Set{S \leadstoin{G} i} = \Set{R \leadstoin{H} i} \, .
\]
For every fixed $r \subset \Gamma$, the family $\Set*{r \leadstoin{H} i}$ for
$i \in V\setminus S$ is independent of $R$ because it depends only on the
orientations of the edges of $H$, while $R$ depends only on the orientations
of the edges of $G$ that go from $S$ to $\Gamma$ -- and these two sets of edges
are disjoint.
Moreover, by the induction hypothesis, the events $\Set*{r \leadstoin{H} i}$,
$i \in V \setminus S$, are positively associated. Since for fixed sets $r$
and $r'$ such that $r' \subset r$,
$\Set{r' \leadsto i} \implies \Set{r \leadsto i}$ for all vertices, we can apply
Lemma~\ref{lemmapositiveassociation} to conclude that the events
$\Set{R \leadsto i}$, $i \in V \setminus S$, are positively associated.

To conclude the proof, note that the events $\Set{S \leadsto i}$ are certain
for $i \in S$ and that the union of a family of positively associated events
and of a family of certain events is still positively related.
\end{proof}

\subsection{Corollaries}

\begin{corollary}
Let $G$ be a finite graph with independently oriented edges.
For any target set $T$, the events $\Set{i \leadsto T}$, $i \in V$,
are positively associated.
\end{corollary}

\begin{proof}
Consider the randomly oriented graph $H$ obtained by reversing the orientation
of the edges of $G$, i.e.\ such that $\Set*{i \toin{H} j} = \Set*{j \toin{G} i}$.
Then for all $i \in V$,
\[
  \Set*{i \leadstoin{G} T} = \Set*{T \leadstoin{H} i}\, , 
\]
and we already know from Theorem~\ref{thmMainResult} that the events
$\Set*{T \leadstoin{H} i}$, $i \in V$, are positively associated.
\end{proof}

\begin{corollary} \label{corInfiniteGraph}
Let $G$ be an infinite graph with independently oriented edges.
Let $f$ and $g$ be increasing, non-negative functions on $\R^V$
that depend only on a finite number of coordinates
(i.e.\ such that there exists a finite set $U \subset V$ and
$\tilde{f} \colon \R^U \to \COInterval{0, +\infty}$ such that
$f = \tilde{f} \circ \varphi$, where $\varphi$ is the canonical
surjection from $\R^V$ to $\R^U$).  Then, for any source set $S$, letting
$\mathbf{X} = (\Indic{S \leadsto i})_{i \in V}$,
\[
  \Expec{f(\mathbf{X}) g(\mathbf{X})} \;\geq\;
  \Expec{f(\mathbf{X})}\, \Expec{g(\mathbf{X})} \, .
\]
\end{corollary}

\begin{proof}
Let $G_n$ be an increasing sequence of finite graphs such that
$G = \bigcup_n G_n$, and for all $i \in V$, let
\[
  X_i^{(n)} \,=\; \Set*[\big]{S \dleadstoin{G_n} i} \,, 
\]
so that $X_i^{(n)} \subset X_i^{(n + 1)}$ and
$X_i = \bigcup_n X_i^{(n)}$. Since the functions $f$ and $g$ are increasing,
so are the sequences $f(\mathbf{X}^{(n)})$ and $g(\mathbf{X}^{(n)})$. Thus,
using Theorem~\ref{thmMainResult} and monotone convergence,
\[
  \Expec{\lim_n f\big(\mathbf{X}^{(n)}\big) g\big(\mathbf{X}^{(n)}\big)}
  \;\geq\; \Expec{\lim_n f\big(\mathbf{X}^{(n)}\big)}\,
  \Expec{\lim_n g\big(\mathbf{X}^{(n)}\big)} \, .
\]
Finally, if $f$ and $g$ depend on a finite number of events $X_i$, then for
every realization of $\mathbf{X}$ we have
$\lim_n f(\mathbf{X}^{(n)}) = f(\mathbf{X})$ and
$\lim_n g(\mathbf{X}^{(n)}) = g(\mathbf{X})$.
\end{proof}

\begin{corollary}[Narayanan, 2016]
For any (possibly infinite) graph with independently oriented edges, for
any source set $S$ and for any two vertices $i$ and $j$,
\[
  \Prob{S \leadsto i, S \leadsto j} \;\geq\;
  \Prob{S \leadsto i}\, \Prob{S \leadsto j}
\]
\end{corollary}

\begin{proof}
Take $f \colon (x_k)_{k \in V} \mapsto x_i$
and $g \colon (x_k)_{k \in V} \mapsto x_j$ in Corollary~\ref{corInfiniteGraph}.
\end{proof}

\begin{corollary} \label{corPoissonCluster}
Let $G$ be a finite graph with independently oriented edges and vertex set~$V$.
For any source set~$S$, let
\[
  N = \sum_{i \in V \setminus S}\! \Indic{S \leadsto i}
\]
denote the size of the oriented percolation cluster of $G$, and set
$\lambda = \Expec{N}$. Then,
\[
  d_{\mathrm{TV}}\big(N,\, \mathrm{Poisson}(\lambda)\big) \;\leq\;
  \min(1, \lambda^{-1})
  \left(\Var{N} - \lambda +
        2 \sum_{i \in V \setminus S}\!\! \Prob{S \leadsto i}^2 \right)\,,
\]
where $d_{\mathrm{TV}}$ denotes the total variation distance.
\end{corollary}

\begin{proof}
This is a direct application of the Stein--Chen method to the positively
related variables $\Indic{S \leadsto i}$, $i \in V \setminus S$ -- see e.g.\
Theorem~4.20 in \cite{RossProbabilitySurveys2011}.
\end{proof}

The interest of Corollary~\ref{corPoissonCluster} is that one only needs a
suitable upper bound on
$\Cov[\normalsize]{\Indic{S \leadsto i},\, \Indic{S \leadsto j}}$ to show that
the size of the oriented percolation cluster is Poissonian.

\section*{Acknowledgments}

I thank my adviser Amaury Lambert for a careful rereading of the first
draft of this manuscript, and Jean-Jil Duchamps for helpful
discussions. I also thank two anonymous reviewers. 

\bibliographystyle{abbrv}
\bibliography{biblio}

\begin{thebibliography}{1}

\bibitem{AhlswedeDaykinProbTheorRF1978}
R.~Ahlswede and D.~E. Daykin.
\newblock An inequality for the weights of two families of sets, their unions
  and intersections.
\newblock {\em Probability Theory and Related Fields}, 43(3):183--185, 1978.

\bibitem{Barbour1992}
A.~D. Barbour, L.~Holst, and S.~Janson.
\newblock {\em Poisson approximation}.
\newblock Oxford Studies in Probability. Oxford University Press, 1992.

\bibitem{DurrettAnnProb1984}
R.~Durrett.
\newblock Oriented percolation in two dimensions.
\newblock {\em The Annals of Probability}, 12(4):999--1040, 1984.

\bibitem{EsaryAnnMathStat1967}
J.~D. Esary, F.~Proschan, and D.~W. Walkup.
\newblock Association of random variables, with applications.
\newblock {\em The Annals of Mathematical Statistics}, 38(5):1466--1474, 1967.

\bibitem{FillAnnApplProb1993}
J.~A. Fill and R.~Pemantle.
\newblock Percolation, first-passage percolation and covering times for
  {R}ichardson's model on the n-cube.
\newblock {\em The Annals of Applied Probability}, 3(2):593--629, 1993.

\bibitem{FortuinCommMathPhy1971}
C.~M. Fortuin, P.~W. Kasteleyn, and J.~Ginibre.
\newblock Correlation inequalities on some partially ordered sets.
\newblock {\em Communications in Mathematical Physics}, 22(2):89--103, 1971.

\bibitem{LinussonArXiv2009}
S.~Linusson.
\newblock A note on correlations in randomly oriented graphs.
\newblock {\em arXiv preprint arXiv:0905.2881}, 2009.

\bibitem{NarayananCombinatProbComput2016}
B.~Narayanan.
\newblock Connections in randomly oriented graphs.
\newblock {\em Combinatorics, Probability and Computing}, pages 1--5, 2016.

\bibitem{RossProbabilitySurveys2011}
N.~Ross.
\newblock Fundamentals of {S}tein's method.
\newblock {\em Probability Surveys}, 8:201--293, 2011.

\end{thebibliography}

\end{document}